\documentclass[11pt]{article}

\usepackage{todonotes}

\usepackage{mathpazo}

\usepackage{amsmath, amssymb, amsthm} 
\usepackage[margin=2.5cm]{geometry}
\usepackage{color, graphicx, soul}
\usepackage{caption}

\usepackage{hyperref}
\hypersetup{
	colorlinks=true,        
	linkcolor=blue,         
	citecolor=blue,         
	urlcolor=blue           
}

\usepackage[overload]{empheq}
\definecolor{myblue}{rgb}{.9, .9, 1}
\newcommand*\mybluebox[1]{%
	\colorbox{myblue}{\hspace{1em}#1\hspace{1em}}
}

\usepackage[capitalize]{cleveref}
\crefname{equation}{}{equations}
\crefname{chapter}{Appendix}{chapters}
\crefname{item}{}{items}
\crefname{figure}{Figure}{figures}
\makeatletter
\def\namedlabel#1#2{\begingroup
	\def\@currentlabel{#2}%
	\label{#1}\endgroup
}
\makeatother

\usepackage{mathtools}

\makeatletter
\def\th@plain{%
	\thm@notefont{}
	\itshape 
}
\def\th@definition{%
	\thm@notefont{}
	\normalfont 
}

\renewenvironment{proof}[1][\proofname]{\par
	\normalfont
	\topsep0\p@\@plus3\p@ \trivlist
	\item[\hskip\labelsep\itshape
	#1\@addpunct{.}]\ignorespaces
}{%
\qed\endtrivlist
}
\makeatother

\newtheorem{theorem}{Theorem}[section]
\newtheorem{lemma}{Lemma}[section]
\newtheorem{corollary}{Corollary}[section]
\newtheorem{proposition}{Proposition}[section]
\newtheorem{fact}{Fact}[section]
\theoremstyle{definition}

\theoremstyle{definition}
\newtheorem{example}{Example}[section]
\theoremstyle{definition}
\newtheorem{remark}{Remark}[section]

\usepackage[shortlabels]{enumitem}
\setlist[enumerate]{nosep}

\allowdisplaybreaks

\newcommand{\menge}[2]{\{{#1}~\big |~{#2}\}} 
 
\newcommand{\Menge}[2]{\left\{{#1}~\Big|~{#2}\right\}}

\newcommand{\scal}[2]{\langle {#1},{#2} \rangle}

\newcommand{\NN}{\ensuremath{\mathbb N}}
\newcommand{\nnn}{\ensuremath{{n\in{\mathbb N}}}}
\newcommand{\RR}{\ensuremath{\mathbb R}}
\newcommand{\RP}{\ensuremath{\mathbb{R}_+}}
\newcommand{\RPP}{\ensuremath{\mathbb{R}_{++}}}

\newcommand{\argmin}{\ensuremath{\operatorname*{argmin}}}

\newcommand{\Fix}{\ensuremath{\operatorname{Fix}}}
\newcommand{\Id}{\ensuremath{\operatorname{Id}}}

\begin{document}

\title{The Douglas--Rachford algorithm\\
	for a hyperplane and a doubleton}

\author{
	Heinz H.\ Bauschke\thanks{
		Mathematics, University of British Columbia, Kelowna, B.C.\ V1V~1V7, Canada. 
		E-mail: \texttt{heinz.bauschke@ubc.ca}.},~
	Minh N.\ Dao\thanks{
		CARMA, University of Newcastle, Callaghan, NSW 2308, Australia.
		E-mail: \texttt{daonminh@gmail.com}.}~~~and~
	Scott B.\ Lindstrom\thanks{
		CARMA, University of Newcastle, Callaghan, NSW 2308, Australia.
		E-mail: \texttt{scott.lindstrom@uon.edu.au}.}
}

\date{April 24, 2018}

\maketitle

\begin{abstract}
The Douglas--Rachford algorithm is a popular algorithm for
solving both convex and nonconvex feasibility problems. 
While its behaviour is settled in the convex inconsistent case,
the general nonconvex inconsistent case is far from being fully
understood. In this paper, we focus on the most simple nonconvex
inconsistent case: when one set is a hyperplane and the other a
doubleton (i.e., a two-point set). We present a characterization of cycling in this
case which --- somewhat surprisingly --- depends on whether 
the ratio of the distance of the points to the hyperplane is
rational or not. Furthermore, we provide closed-form expressions
as well as several concrete examples which illustrate the
dynamical richness of this algorithm.
\end{abstract}

{\small
\noindent
{\bfseries 2010 Mathematics Subject Classification:}
{Primary:  
47H10, 
49M27; 
Secondary:  
65K05, 
65K10, 
90C26.
} 

\noindent {\bfseries Keywords:}
closed-form expressions,
cycling,
Douglas--Rachford algorithm,
feasibility problem,
finite set,
hyperplane,
method of alternating projections,
projector,
reflector
}

\section{Introduction}
\label{s:intro}

The Douglas--Rachford (DR) algorithm \cite{DR56} is a popular algorithm for
finding minimizers of the sum of two functions, defined on a
real Hilbert space and possibly nonsmooth.
Its convergence properties are fairly well understood in the 
case when the function are convex; see
\cite{LM79},
\cite{EB92},
\cite{Com04}, 
\cite{BCL04}, 
\cite{BDM16},
and \cite{BM17}.
When specialized to
indicator functions, 
the DR algorithm aims to solve a feasibility 
problem. 

The \emph{goal} of this paper is to analyze an instructive --- and 
perhaps the most simple ---
nonconvex setting: when one set is a hyperplane and the other
is a doubleton (i.e., it consists of
just two distinct points). 
Our analysis reveals interesting dynamic behaviour whose
\emph{periodicity} depends on whether or not a certain ratio of distances
is rational (Theorem~\ref{t:2points}).
We also provide \emph{explicit closed-form expressions} for the
iterates in various circumstances (Theorem~\ref{t:closedform}). 
Our work can be regarded as complementary to 
the recently rapidly growing body of works on the DR 
algorithm in
nonconvex settings including 
\cite{ERT07},
\cite{BS11},
\cite{HL13},
\cite{BN14},
\cite{ABT16}, 
\cite{Pha16},
and \cite{DT17}.

The remainder of the paper is organized as follows.
In Section~\ref{s:setup},
we recall the necessary background material to start our
analysis. 
The case when one set contains not just $2$ but finitely many
points is considered in Section~\ref{s:finite}. 
Section~\ref{s:cycling} provides a characterization of when
cycling occurs, while
Section~\ref{s:closed-form} presents closed-form expressions and
various examples.
We conclude the paper with Section~\ref{s:conclusion}.

\section{The set up}
\label{s:setup}

Throughout we assume that
\begin{empheq}[box=\mybluebox]{equation}
\text{$X$ is a finite-dimensional real Hilbert space}
\end{empheq}
with inner product $\scal{\cdot}{\cdot}$ and induced norm $\|\cdot\|$, and 
\begin{empheq}[box=\mybluebox]{equation}
\text{$A$ and $B$ are nonempty closed subsets of $X$}.
\end{empheq}
To solve the feasibility problem 
\begin{empheq}[box=\mybluebox]{equation}
\label{e:prob}
\text{find a point in $A\cap B$},
\end{empheq}  
we employ the \emph{Douglas--Rachford algorithm} (also called \emph{averaged alternating reflections}) 
that uses the \emph{DR operator}, associated with the ordered pair $(A, B)$,
\begin{empheq}[box=\mybluebox]{equation}
T :=\tfrac{1}{2}(\Id+R_BR_A)
\end{empheq}
to generate a \emph{DR sequence} $(x_n)_\nnn$ with starting point $x_0 \in X$ by
\begin{empheq}[box=\mybluebox]{equation}
\label{e:DRAseq}
(\forall\nnn)\quad x_{n+1} \in Tx_n,
\end{empheq}
where $\Id$ is the identity operator, $P_A$ and $P_B$ are the projectors, 
and $R_A :=2P_A -\Id$ and $R_B :=2P_B -\Id$ are the reflectors with respect to $A$ and $B$, respectively.
Here the projection $P_Ax$ of a point $x \in X$ is the nearest point of $x$ in the set $A$, i.e.,
\begin{equation}
P_Ax :=\argmin_{a \in A} \|x -a\| =\menge{a \in A}{\|x -a\| =d_A(x)},
\end{equation} 
where $d_A(x) := \min_{a \in A} \|x -a\|$ is the distance from $x$ to the set $A$.
Note from \cite[Corollary~3.15]{BC17} that 
closedness of the set $A$ is necessary and sufficient for $A$ being proximinal, 
i.e., $(\forall x\in X)$ $P_Ax\neq \varnothing$.
According to \cite[Theorem~3.16]{BC17},
if $A$ and $B$ are convex, then $P_A$, $P_B$ and hence $T$ are single-valued.
We also note that
\begin{equation}
(\forall x \in X)\quad Tx =\tfrac{1}{2}(\Id+R_BR_A)x =\menge{x -a +P_B(2a -x)}{a \in P_Ax},
\end{equation}
and if $P_A$ is single-valued then
\begin{equation}
\label{e:PAsingle}
T =\tfrac{1}{2}(\Id+R_BR_A) =\Id -P_A +P_BR_A.
\end{equation}
For further information on the DR algorithm 
in the classical case
(when $A$ and $B$ are both convex),
see 
\cite{LM79},
\cite{Com04}, 
\cite{BCL04}, 
\cite{BM17}, 
and \cite{BDNP16b}.
Results complementary to the rapidly increasing body of works on the
DR algorithm in nonconvex settings can be found in 
\cite{BN14},
\cite{Pha16},
\cite{DP16},
\cite{BLSSS17},
\cite{LLS17}, 
\cite{LSS17},
\cite{DP17}, and the references therein. 


The notation and terminology used is standard and follows, e.g., \cite{BC17}.
The nonnegative integers are $\NN$, the positive integers are $\NN^*$, and the real numbers are $\RR$, 
while $\RP := \menge{x \in \RR}{x \geq 0}$ and $\RPP := \menge{x \in \RR}{x >0}$.
We are now ready to start deriving the results we announced in
Section~\ref{s:intro}. 

\section{Hyperplane and finitely many points}
\label{s:finite}

We focus on the case when $B$ is a finite set, and we start with the following observation.
\begin{lemma}
\label{l:notcvg}
Suppose that $A$ is convex, that $B$ is finite, and that $A\cap B =\varnothing$. 
Let $(x_n)_\nnn$ be a DR sequence with respect to $(A, B)$.
Then $(x_n)_\nnn$ is not convergent.
\end{lemma}
\begin{proof}
Since $A$ is convex, $P_A$ is single-valued and continuous on $X$.
By \eqref{e:PAsingle}, $T =\frac{1}{2}(\Id +R_BR_A) =\Id -P_A +P_BR_A$, and hence
\begin{equation}
(\forall\nnn)\quad b_n :=x_{n+1} -x_n +P_Ax_n \in P_BR_Ax_n \subseteq B.
\end{equation}
Suppose that $x_n \to x \in X$. Then $b_n \to P_Ax$.
But $(b_n)_\nnn$ lies in $B$ and $B$ is finite, there exists $n_0 \in \NN$ such that
$(\forall n\geq n_0)$ $b_n =b \in B$. 
We obtain $P_Ax =b \in A\cap B$, which contradicts the assumption that $A\cap B =\varnothing$.
\end{proof}

From here onwards, we assume that $A$ is a hyperplane and $B$ is
a finite subset of $X$ containing $m$ pairwise distinct vectors; 
more specifically,
\begin{subequations}
\begin{empheq}[box=\mybluebox]{equation}
A= \{u\}^\perp \quad\text{with}\quad u\in X, \|u\|= 1
\end{empheq}
and 
\begin{empheq}[box=\mybluebox]{equation}
\label{e:B}
B= \{b_1, \dots, b_m\}\subseteq X \quad\text{with}\quad \scal{b_1}{u}\leq \cdots \leq \scal{b_m}{u}.
\end{empheq}
\end{subequations}

\begin{fact}
\label{f:A}
Let $x\in X$. Then the following hold:
\begin{enumerate}
\item\label{f:A_P} 
$P_Ax= x- \scal{x}{u}u$.
\item\label{f:A_R} 
$R_Ax= x- 2\scal{x}{u}u$.
\item\label{f:A_d} 
$d_A(x)= |\scal{x}{u}|$.
\end{enumerate}
\end{fact}
\begin{proof}
This follows from \cite[Example~2.4(i)]{BD17} with noting that 
$R_Ax= 2P_Ax- x$ and that $d_A(x)= \|x- P_Ax\|$.
\end{proof}

Let $(x_n)_\nnn$ be a DR sequence with respect to $(A, B)$ with starting point $x_0\in X$. 
Since $P_A$ is single-valued, we derive from \eqref{e:PAsingle} that 
\begin{equation}
(\forall n\in \NN^*)\quad 
x_n- x_{n-1}+ P_Ax_{n-1}\in Tx_{n-1}- x_{n-1}+ P_Ax_{n-1}= P_BR_Ax_{n-1}\subseteq B.
\end{equation}
Let us set
\begin{empheq}[box=\mybluebox]{equation}
\label{e:b_kn}
(\forall n\in \NN^*)\quad b_{k(n)}:= x_n- x_{n-1}+ P_Ax_{n-1}\in P_BR_Ax_{n-1}\subseteq B 
\text{~with~} k(n) \in \{1, \dots, m\}.
\end{empheq}
The following lemma shows that the subsequence $(x_n)_{n\in
\NN^*}$ lies in the union of the lines through the points in $B$
with a common direction vector $u$. 

\begin{lemma}
\label{l:lines}
For every $n\in \NN^*$, 
\begin{equation}
x_n= \scal{x_{n-1}}{u}u+ b_{k(n)} 
\quad\text{and}\quad \scal{x_n}{u}= \scal{x_{n-1}}{u}+ \scal{b_{k(n)}}{u}, 
\end{equation}
where $k(n)\in \{1, \dots, m\}$. Consequently, the subsequence $(x_n)_{n\in \NN^*}$ lies in the union of finitely many (affine) lines:
\begin{equation}
B+ \RR u=\bigcup_{b\in B} (b+ \RR u)= \menge{b+ \lambda u}{b\in B, \lambda\in \RR}.
\end{equation}  
\end{lemma}

\begin{proof}
By combining \eqref{e:b_kn} with Fact~\ref{f:A}\ref{f:A_P},
\begin{equation}
\label{e:x+}
(\forall n\in \NN^*)\quad x_n= x_{n-1}- P_Ax_{n-1}+ b_{k(n)}= \scal{x_{n-1}}{u}u+ b_{k(n)}.
\end{equation}
Taking the inner product with $u$ yields
\begin{equation}
\label{e:x+,u}
(\forall n\in \NN^*)\quad \scal{x_n}{u}= \scal{\scal{x_{n-1}}{u}u+ b_{k(n)}}{u}= \scal{x_{n-1}}{u}+ \scal{b_{k(n)}}{u},
\end{equation} 
which completes the proof.
\end{proof}

\begin{proposition}
\label{p:mpoints}
Exactly one of the following holds. 
\begin{enumerate}
\item 
\label{p:mpoints_finite}
$B$ is contained in one of 
the two closed halfspaces induced by $A$.   
Then either {\rm (a)} the sequence $(x_n)_\nnn$ converges finitely to
a point $x \in \Fix T$ and $P_Ax \in A\cap B$, 
or {\rm (b)} $A\cap B= \varnothing$ and $\|x_n\| \to +\infty$ in which
case $(P_Ax_n)_\nnn$ converges finitely to a best approximation
solution $a \in A$ relative to $A$ and $B$ in the sense that
$d_B(a)= \min d_B(A)$.
\item 
\label{p:mpoints_bounded}
$B$ is not contained in one of the two closed halfspaces induced by $A$.  
Then the sequence $(x_n)_\nnn$ is bounded. If
additionally $A\cap B= \varnothing$, then $(x_n)_\nnn$ is not
convergent and  
\begin{equation}
(\forall\nnn)\quad \|x_n- x_{n+1}\| \geq \min d_A(B)> 0.
\end{equation}
\end{enumerate}
\end{proposition}
\begin{proof}
\ref{p:mpoints_finite}: This follows from \cite[Theorem~7.5]{BD17}.

\ref{p:mpoints_bounded}: Since $B$ is not a subset of one of two closed halfspaces induced by $A$, it follows from \eqref{e:B} that 
\begin{equation}
\label{e:b1bm}
\scal{b_1}{u}< 0< \scal{b_m}{u}.
\end{equation}
Combining Fact~\ref{f:A}\ref{f:A_R} with Lemma~\ref{l:lines}  yields
\begin{subequations}
\begin{align}
\label{e:RAx+}
(\forall n\in \NN^*)\quad R_Ax_n&= x_n- 2\scal{x_n}{u}u \\
&= \big( \scal{x_{n-1}}{u}u+ b_{k(n)} \big)- \big( \scal{x_{n-1}}{u}+ \scal{b_{k(n)}}{u} \big)u- \scal{x_n}{u}u \\ 
&= -(\scal{x_n}{u}+\scal{b_{k(n)}}{u})u+ b_{k(n)}.
\end{align}
\end{subequations}
For any $n\in \NN^*$ and any distinct indices $i, j\in \{1, \dots, m\}$, we have the following equivalences:
\begin{subequations}
\label{e:compare}
\begin{align}
&\|b_i- R_Ax_n\|\leq \|b_j- R_Ax_n\| \\ 
\Leftrightarrow{} &\|(\scal{x_n}{u}+\scal{b_{k(n)}}{u})u+ (b_i-b_{k(n)})\|^2 
\leq \|(\scal{x_n}{u}+\scal{b_{k(n)}}{u})u+ (b_j-b_{k(n)})\|^2 \\ 
\Leftrightarrow{} &\|b_i-b_{k(n)}\|^2- \|b_j-b_{k(n)}\|^2  
\leq 2(\scal{x_n}{u}+\scal{b_{k(n)}}{u})\scal{b_j-b_i}{u} \\
\Leftrightarrow{} &\begin{cases}
\scal{x_n}{u} \geq \beta_{i,j,n} &\text{if~} \scal{b_i}{u}< \scal{b_j}{u}, \\
\|b_i-b_{k(n)}\|\leq \|b_j-b_{k(n)}\| &\text{if~} \scal{b_i}{u}= \scal{b_j}{u}, \\
\scal{x_n}{u} \leq \beta_{i,j,n} &\text{if~} \scal{b_i}{u}> \scal{b_j}{u},
\end{cases}
\end{align}
\end{subequations} 
where 
\begin{equation}
\beta_{i,j,n}:= \frac{\|b_i-b_{k(n)}\|^2- \|b_j-b_{k(n)}\|^2}{2\scal{b_j-b_i}{u}} -\scal{b_{k(n)}}{u}.
\end{equation}
We shall now show that $(\scal{x_n}{u})_\nnn$ is bounded above. Setting 
\begin{equation}
r:= \max\menge{k\in \{1, \dots, m\}}{\scal{b_k}{u}= \scal{b_1}{u}},
\end{equation}
we see that $r< m$ due to \eqref{e:b1bm} and that, by \eqref{e:B}, 
\begin{equation}
\label{e:b1br}
\scal{b_1}{u}= \cdots= \scal{b_r}{u}< \scal{b_{r+1}}{u}\leq \cdots\leq \scal{b_m}{u}.
\end{equation}
Now let $n\in \NN^*$ and set 
\begin{equation}
I(n):= \menge{i\in \{1, \dots, r\}}{(\forall j\in \{1, \dots, r\})\quad \|b_i-b_{k(n)}\|\leq \|b_j-b_{k(n)}\|}.
\end{equation}
Then $I(n)= \{k(n)\}$ whenever $k(n)\in \{1, \dots, r\}$ and, by \eqref{e:compare}, 
\begin{equation}
\label{e:pre-compare}
(\forall i\in I(n))(\forall j\in \{1, \dots, r\})\quad \|b_i- R_Ax_n\|\leq \|b_j- R_Ax_n\|. 
\end{equation} 
Define
\begin{equation}
\beta_n:= \max\menge{\beta_{i,j,n}}{i\in I(n), j\in \{r+1, \dots, m\}}.
\end{equation}
If $\scal{x_n}{u}> \beta_n$, then \eqref{e:B} and \eqref{e:compare} yield
\begin{equation}
(\forall i\in I(n))(\forall k\in \{r+1, \dots, m\})\quad \|b_i- R_Ax_n\|< \|b_k- R_Ax_n\|, 
\end{equation}
which together with \eqref{e:pre-compare} implies that $k(n+1)\in I(n)\subseteq \{1, \dots, r\}$ and, 
by \eqref{e:x+,u}, \eqref{e:b1bm} and \eqref{e:b1br}, 
\begin{equation}
\label{e:decrease}
\scal{x_{n+1}}{u}= \scal{x_n}{u}+ \delta \quad\text{with}\quad \delta:= \scal{b_{k(n+1)}}{u}= \scal{b_1}{u}< 0.
\end{equation}
Noting that \eqref{e:decrease} holds whenever $\scal{x_n}{u}> \beta_n$ and that the sequence $(\beta_n)_\nnn$ is bounded since the set $\menge{\beta_{i,j,n}}{i\in I(n), j\in \{r+1, \dots, m\}, n\in \NN^*}$ is finite, we deduce that $(\scal{x_n}{u})_\nnn$ is bounded above. By a similar argument, $(\scal{x_n}{u})_\nnn$ is also bounded below. Combining with \eqref{e:x+}, we get boundedness of $(x_n)_\nnn$.

Finally, if $A\cap B= \varnothing$, then, by
Lemma~\ref{l:notcvg}, $(x_n)_\nnn$ is not convergent and, by the 
Cauchy--Schwarz inequality, Lemma~\ref{l:lines}, and Fact~\ref{f:A}\ref{f:A_d}, 
\begin{equation}
(\forall\nnn)\quad \|x_{n+1}- x_n\|\geq |\scal{x_{n+1}-x_n}{u}|=
|\scal{b_{k(n+1)}}{u}|= d_A(b_{k(n+1)})\geq \min d_A(B)> 0.
\end{equation}
The proof is complete.
\end{proof}

\section{Hyperplane and doubleton: characterization of cycling}
\label{s:cycling}

From now on, 
we assume that $B$ is a doubleton where 
the two points do not belong to the same closed halfspace induced by $A$; more precisely,
\begin{empheq}[box=\mybluebox]{equation}
B= \{b_1, b_2\}\subseteq X \quad\text{with}\quad \scal{b_1}{u}< 0< \scal{b_2}{u}.
\end{empheq}
Set 
\begin{empheq}[box=\mybluebox]{equation}
\label{e:32}
\beta_1:= \scal{b_1}{u}< 0,\quad \beta_2:= \scal{b_2}{u}>0, \quad\text{and}\quad \beta:= \frac{\|b_1-b_2\|^2}{2(\beta_1-\beta_2)}= -\frac{\|b_1-b_2\|^2}{2\scal{b_2-b_1}{u}}< 0.
\end{empheq}

\begin{proposition}
\label{p:2points}
The following holds for the DR sequence $(x_n)_\nnn$. 
\begin{enumerate}
\item 
\label{p:2points_bounded}
$(x_n)_\nnn$ is bounded but not convergent with 
\begin{equation}
(\forall\nnn)\quad \|x_n- x_{n+1}\| \geq \min\{d_A(b_1), d_A(b_2)\}> 0.
\end{equation}
\item 
\label{p:2points_x+}
For every $n\in \NN^*$, 
\begin{equation}
\label{e:x,u}
x_n= \scal{x_{n-1}}{u}u+ b_{k(n)} 
\quad\text{and}\quad \scal{x_n}{u}= \scal{x_{n-1}}{u}+ \scal{b_{k(n)}}{u}, 
\end{equation}
where $k(n)\in \{1, 2\}$ and where   
\begin{subequations}
\label{e:kn+}
\begin{align}
k(n)= 1\ \&\ \scal{x_n}{u}> \beta- \scal{b_1}{u} &\implies k(n+1)= 1, \\
k(n)= 1\ \&\ \scal{x_n}{u}< \beta- \scal{b_1}{u} &\implies k(n+1)= 2, \\
k(n)= 2\ \&\ \scal{x_n}{u}> -\beta- \scal{b_2}{u} &\implies k(n+1)= 1, \\
k(n)= 2\ \&\ \scal{x_n}{u}< -\beta- \scal{b_2}{u} &\implies k(n+1)= 2.
\end{align}
\end{subequations}
\item 
\label{p:2points_coeffs}
There exist increasing (a.k.a.\ ``nondecreasing'') 
sequences $(l_{1,n})_\nnn$ and $(l_{2,n})_\nnn$ in $\NN$ such that 
\begin{equation}
(\forall\nnn)\quad \scal{x_n}{u} =\scal{x_0}{u} +l_{1,n}\scal{b_1}{u} +l_{2,n}\scal{b_2}{u}
\quad\text{and}\quad l_{1,n} +l_{2,n} =n.
\end{equation}
Moreover, 
\begin{equation}
\frac{l_{1,n}}{n}\to \frac{\scal{b_2}{u}}{\scal{b_2-b_1}{u}}\in \left]0, 1\right[
\quad\text{and}\quad 
\frac{l_{2,n}}{n}\to \frac{\scal{b_1}{u}}{\scal{b_1-b_2}{u}}\in \left]0, 1\right[
\quad\text{as~} n\to +\infty.
\end{equation}
\end{enumerate}
\end{proposition}
\begin{proof}
\ref{p:2points_bounded}: By assumption, $b_1, b_2 \notin A$, and hence $A\cap B =\varnothing$. The conclusion follows from Proposition~\ref{p:mpoints}\ref{p:mpoints_bounded}.

\ref{p:2points_x+}: We get \eqref{e:x,u} from Lemma~\ref{l:lines}. 
The equivalences \eqref{e:compare} in the 
proof of Proposition~\ref{p:mpoints}\ref{p:mpoints_bounded}
state 
\begin{equation}
\|b_1- R_Ax_n\|\leq \|b_2- R_Ax_n\| \Leftrightarrow{} 
\scal{x_n}{u}\geq \frac{\|b_1-b_{k(n)}\|^2- \|b_2-b_{k(n)}\|^2}{2\scal{b_2-b_1}{u}} -\scal{b_{k(n)}}{u},
\end{equation}
which implies \eqref{e:kn+}.

\ref{p:2points_coeffs}: Using \eqref{e:x,u}, we find increasing 
sequences $(l_{1,n})_\nnn$ and $(l_{2,n})_\nnn$ in $\NN$ such that
\begin{equation}
\label{e:xnx0}
(\forall\nnn)\quad \scal{x_n}{u} =\scal{x_0}{u} +l_{1,n}\scal{b_1}{u} +l_{2,n}\scal{b_2}{u}
\end{equation} 
and that
\begin{equation}
(\forall\nnn)\quad l_{1,n} +l_{2,n} =n.
\end{equation}
Combining with \ref{p:2points_bounded}, we obtain that 
\begin{equation}
l_{1,n}\scal{b_1}{u} +(n -l_{1,n})\scal{b_2}{u} =l_{1,n}\scal{b_1}{u} +l_{2,n}\scal{b_2}{u} 
=\scal{x_n}{u} -\scal{x_0}{u}
\end{equation}
is bounded. It follows that
\begin{equation}
\frac{l_{1,n}}{n}\scal{b_1-b_2}{u}+ \scal{b_2}{u}\to 0 \quad\text{as~} n\to +\infty,
\end{equation}
which yields
\begin{equation}
\frac{l_{1,n}}{n}\to \frac{\scal{b_2}{u}}{\scal{b_2-b_1}{u}}\in \left]0, 1\right[ \quad\text{and}\quad 
\frac{l_{2,n}}{n}= 1- \frac{l_{1,n}}{n}\to \frac{-\scal{b_1}{u}}{\scal{b_2 -b_1}{u}}\in \left]0, 1\right[
\end{equation}
as $n\to +\infty$.
\end{proof}

\begin{theorem}[cycling and rationality] 
\label{t:2points}
The DR sequence $(x_n)_\nnn$ cycles after a certain number of
steps regardless of the starting point 
if and only if $d_A(b_1)/d_A(b_2)\in \mathbb{Q}$.
\end{theorem}
\begin{proof}
First, by Fact~\ref{f:A}\ref{f:A_d}, $d_A= |\scal{\cdot}{u}|$, which yields
\begin{equation}
\label{e:dAb1b2}
d_A(b_1)= -\scal{b_1}{u} \quad\text{and}\quad d_A(b_2)= \scal{b_2}{u}.
\end{equation} 
We also note from Proposition~\ref{p:2points}\ref{p:2points_bounded}--\ref{p:2points_x+} that 
\begin{equation}
\label{e:bounded}
\text{$(|\scal{x_n}{u}|)_\nnn$ is bounded}, 
\end{equation}
that 
\begin{equation}
\label{e:xn}
(\forall n\in \NN^*)\quad x_n= \scal{x_{n-1}}{u}u+ b_{k(n)},
\end{equation}
and that 
\begin{equation}
\label{e:xnu}
(\forall n\in \NN^*)\quad \scal{x_n}{u}= \scal{x_{n-1}}{u}+ \scal{b_{k(n)}}{u}, 
\end{equation}
where $k(n)\in \{1, 2\}$. 

``$\Leftarrow$'': Assume that $d_A(b_1)/d_A(b_2)\in \mathbb{Q}$. Then there exist $q_1, q_2\in \NN^*$ such that $q_1d_A(b_1) =q_2d_A(b_2)$, or equivalently (using \eqref{e:dAb1b2}),
\begin{equation}
\label{e:suff}
q_1\scal{b_1}{u} +q_2\scal{b_2}{u} =0.
\end{equation}
It follows from Proposition~\ref{p:2points}\ref{p:2points_coeffs} that  
\begin{equation}
\label{e:xnx0'}
(\forall\nnn)\quad \scal{x_n}{u} =\scal{x_0}{u} +l_{1,n}\scal{b_1}{u} +l_{2,n}\scal{b_2}{u}
\end{equation} 
with $(l_{1,n}, l_{2,n})\in \NN^2$.
By \eqref{e:suff}, whenever $l_{1,n} \geq q_1$ and $l_{2,n} \geq q_2$, we have
\begin{equation}
\scal{x_n}{u} = \scal{x_0}{u} +(l_{1,n} -q_1)\scal{b_1}{u} +(l_{2,n} -q_2)\scal{b_2}{u}.
\end{equation}
We can thus restrict to considering the sequences $l_{1,n}',l_{2,n}'$ satisfying \eqref{e:xnx0'} 
and also the additional stipulation that $l_{1,n}' <q_1$ or $l_{2,n}' <q_2$. 
Then $l_{1,n}'\scal{b_1}{u}$ or $l_{2,n}'\scal{b_2}{u}$ is bounded. 
This together with \eqref{e:bounded} and \eqref{e:xnx0'} implies that 
both $l_{1,n}'\scal{b_1}{u}$ and $l_{2,n}'\scal{b_2}{u}$ are bounded, and so are $l_{1,n}'$ and $l_{2,n}'$. 
Hence, there exist $L_1, L_2 \in \NN$ such that 
\begin{equation}
(\forall\nnn)\quad 0 \leq l_{1,n}' \leq L_1 \quad\text{and}\quad 0 \leq l_{2,n}' \leq L_2.
\end{equation}
By combining with \eqref{e:xn} and \eqref{e:xnx0'}, $(\forall n\in \NN^*)$ $x_n\in S$, where
\begin{equation}
S :=\menge{\scal{x_0}{u}u +l_1'\scal{b_1}{u}u +l_2'\scal{b_2}{u}u +b_k}
{l_1' =0, \dots, L_1,\ l_2' =0, \dots, L_2,\ k =1, 2}.
\end{equation} 
Since $S$ is a finite set, there exist $n_0\in \NN$ and $m\in \NN^*$ such that $x_{n_0} =x_{n_0+m}$.
It follows that the sequence $(x_n)_\nnn$ cycles between $m$ points $x_{n_0}, \dots, x_{n_0+m-1}$ from $n_0$ onwards. 

``$\Rightarrow$'': Assume that $(x_n)_\nnn$ cycles between $m$ points from $n_0 \in \NN$ onwards, i.e., $(\forall n \geq n_0)$ $x_{n+m} =x_n$. 
By \eqref{e:xnu}, 
\begin{equation}
\scal{x_{n_0}}{u} +\sum_{n=n_0}^{n_0+m-1} \scal{b_{k(n)}}{u} =\scal{x_{n_0}}{u}.
\end{equation}
There thus exist $q_1, q_2 \in \NN$ such that $q_1 +q_2 =m >0$ and $q_1\scal{b_1}{u} +q_2\scal{b_2}{u} =0$.
Combining with \eqref{e:dAb1b2} implies that $q_1, q_2\neq 0$ and that $d_A(b_1)/d_A(b_2)= q_2/q_1\in \mathbb{Q}$. 
\end{proof}

\section{Hyperplane and doubleton: closed-form expressions}
\label{s:closed-form}

In this final section, we refine the previously considered case
with the aim of obtaining \emph{closed-form} expressions for the
terms of the DR sequence $(x_n)_\nnn$. 

Recall from Proposition~\ref{p:2points}\ref{p:2points_x+} that
\begin{equation}
\label{e:x,u'}
(\forall n\in \NN^*)\quad x_n= \scal{x_{n-1}}{u}u+ b_{k(n)} 
\quad\text{and}\quad \scal{x_n}{u}= \scal{x_{n-1}}{u}+ \scal{b_{k(n)}}{u}, 
\end{equation}
where $k(n)\in \{1, 2\}$ and where
\begin{subequations}
\begin{align}
k(n)= 1\ \&\ \scal{x_n}{u}> \beta-\beta_1 &\implies k(n+1)= 1, \label{e:kn11} \\
k(n)= 1\ \&\ \scal{x_n}{u}< \beta-\beta_1 &\implies k(n+1)= 2, \label{e:kn12} \\
k(n)= 2\ \&\ \scal{x_n}{u}> -\beta-\beta_2 &\implies k(n+1)= 1. \label{e:kn21}
\end{align}
\end{subequations}
We note here that if $k(n)= 1$ and $\scal{x_n}{u}=
\beta-\beta_1$, then both $1$ and $2$ are acceptable values for
$k(n+1)$; 
for the sake of simplicity, we choose $k(n+1)= 2$ in this case. 
Define
\begin{subequations}
\begin{empheq}[box=\mybluebox]{align}
S_1&:= \Menge{x_n}{n\in \NN^*,\ k(n)= 1,\ \scal{x_n}{u}\in \left]\beta, \beta+\beta_2\right]},\\
S_2&:= \Menge{x_n}{n\in \NN^*,\ k(n)= 2,\ \scal{x_n}{u}\in \left]\beta+\beta_2, \beta-\beta_1+\beta_2\right]}.
\end{empheq}
\end{subequations}

\begin{proposition}
\label{p:segments}
Let $n\in \NN^*$. Then the following hold:
\begin{enumerate}
\item 
\label{p:segments_11}
If $k(n)= 1$ and $\scal{x_n}{u}\in \left]\beta-\beta_1, \beta+\beta_2\right]$, then 
\begin{equation}
k(n+1)= 1 \quad\text{and}\quad \scal{x_{n+1}}{u}= \scal{x_n}{u}+ \beta_1\in \left]\beta, \beta+\beta_2\right].
\end{equation}
\item 
\label{p:segments_12}
If $k(n)= 1$ and $\scal{x_n}{u}\in \left]\beta, \beta-\beta_1\right]$, then 
\begin{equation}
k(n+1)= 2 \quad\text{and}\quad \scal{x_{n+1}}{u}= \scal{x_n}{u}+ \beta_2\in \left]\beta+\beta_2, \beta-\beta_1+\beta_2\right].
\end{equation}
\item
\label{p:segments_21}
If $k(n)= 2$, $\scal{x_n}{u}\in \left]\beta+\beta_2,
\beta-\beta_1+\beta_2\right]$ and 
$\beta+\beta_2\geq 0$, then 
\begin{equation}
k(n+1)= 1 \quad\text{and}\quad \scal{x_{n+1}}{u}= \scal{x_n}{u}+ \beta_1\in \left]\beta+\beta_1+\beta_2, \beta+\beta_2\right]\subseteq \left]\beta, \beta+\beta_2\right]. 
\end{equation}
\end{enumerate}
Consequently, 
\begin{equation}
\big(\beta+\beta_2\geq 0
\;\text{and}\;
x_n\in S_1\cup S_2\big) \implies x_{n+1}\in S_1\cup S_2.
\end{equation}
\end{proposition}
\begin{proof}
Notice from \eqref{e:x,u'} that
\begin{equation}
\label{e:xu+}
(\forall\nnn)\quad \scal{x_{n+1}}{u}= \scal{x_n}{u}+ \scal{b_{k(n+1)}}{u}.
\end{equation}

\ref{p:segments_11}: Combine \eqref{e:kn11} and \eqref{e:xu+} 
while noting that $\beta+\beta_1+\beta_2< \beta+\beta_2$ by
\eqref{e:32}. 

\ref{p:segments_12}: Combine \eqref{e:kn12} and \eqref{e:xu+}.

\ref{p:segments_21}: 
By \eqref{e:32} and the Cauchy--Schwarz inequality, 
we obtain 
\begin{equation}
0< \beta_2- \beta_1= \scal{b_2-b_1}{u}\leq \|b_2-b_1\|\|u\|= \|b_2-b_1\|
\end{equation}
and 
\begin{equation}
\label{e:inequality}
\beta_2- \beta_1\leq \frac{\|b_2-b_1\|^2}{\beta_2-\beta_1}= -2\beta.
\end{equation}
Now assume that $\beta+ \beta_2\geq 0$. Then $\beta_1+ \beta_2\geq (2\beta+ \beta_2)+ \beta_2= 2(\beta+ \beta_2)\geq 0$, and hence 
$\left]\beta+\beta_1+\beta_2, \beta+\beta_2\right]\subseteq \left]\beta, \beta+\beta_2\right]$.
It follows from $\scal{x_n}{u}> \beta+\beta_2\geq 0$ that $\scal{x_n}{u}> -\beta-\beta_2$. Now use \eqref{e:kn21} and \eqref{e:xu+}. 

Finally, assume that $x_n\in S_1\cup S_2$. If $x_n\in S_2$, then we have from \ref{p:segments_21} that $x_{n+1}\in S_1$. If $x_n\in S_1$ and $\scal{x_n}{u}\in \left]\beta, \beta-\beta_1\right]$, then, by \ref{p:segments_12}, $x_{n+1}\in S_2$. If $x_n\in S_1$ and $\scal{x_n}{u}\in \left]\beta-\beta_1, \beta+\beta_2\right]$, then $x_{n+1}\in S_1$ due to \ref{p:segments_11}. Altogether, $x_{n+1}\in S_1\cup S_2$. 
\end{proof}

\begin{theorem}[closed-form expressions]
\label{t:closedform}
Suppose that $\beta+\beta_2\geq 0$ and that $x_1\in S_1\cup S_2$. 
Then 
\begin{subequations}
\label{e:xu}
\begin{align}
(\forall n\in \NN^*)\quad \scal{x_n}{u}
&= \scal{x_0}{u}+ n\beta_1+ \left\lfloor \frac{-\scal{x_0}{u}+\beta-(n+1)\beta_1+\beta_2}{\beta_2-\beta_1} \right\rfloor (\beta_2- \beta_1) \\
&= \scal{x_0}{u}- \left\lfloor \frac{-\scal{x_0}{u}+\beta-\beta_1-(n-1)\beta_2}{\beta_2-\beta_1} \right\rfloor \beta_1 \notag \\ 
&\hspace{3.5cm}+ \left\lfloor \frac{-\scal{x_0}{u}+\beta-(n+1)\beta_1+\beta_2}{\beta_2-\beta_1} \right\rfloor \beta_2
\end{align}
\end{subequations}
and 
\begin{equation}
\label{e:x}
(\forall n\in \NN^*)\quad x_n= \scal{x_{n-1}}{u}u+ b_{k(n)},  
\end{equation}
where
\begin{subequations}
\label{e:kn}
\begin{align}
(\forall n\in \NN^*)\quad 
k(n)&= \begin{cases}
1 &\text{if~} \scal{x_n}{u}\leq \beta+\beta_2, \\
2 &\text{if~} \scal{x_n}{u}> \beta+\beta_2
\end{cases} \\
&= \left\lfloor \frac{-\scal{x_0}{u}+\beta-(n+1)\beta_1+\beta_2}{\beta_2-\beta_1} \right\rfloor- \left\lfloor \frac{-\scal{x_0}{u}+\beta-n\beta_1+\beta_2}{\beta_2-\beta_1} \right\rfloor+ 1.
\end{align}
\end{subequations}
\end{theorem}
\begin{proof}
Note that \eqref{e:x} follows from \eqref{e:x,u}. 
According to Proposition~\ref{p:2points}\ref{p:2points_coeffs}, 
\begin{equation}
\label{e:xnx0''}
(\forall\nnn)\quad \scal{x_n}{u}= \scal{x_0}{u}+ (n-l_n)\beta_1+ l_n\beta_2 \quad\text{with}\quad l_n\in \NN.
\end{equation}
Since $x_1\in S_1\cup S_2$, Proposition~\ref{p:segments} yields 
\begin{equation}
\label{e:xnS1S2}
(\forall n\in \NN^*)\quad x_n\in S_1\cup S_2.
\end{equation}
Let $n\in \NN^*$. 
It follows from \eqref{e:32} and 
\eqref{e:xnS1S2} that $\scal{x_n}{u}\in \left]\beta, \beta-\beta_1+\beta_2\right]$, which, combined with \eqref{e:xnx0''}, gives
\begin{equation}
\frac{-\scal{x_0}{u}+\beta-(n+1)\beta_1+\beta_2}{\beta_2-\beta_1}- 1<l_n\leq \frac{-\scal{x_0}{u}+\beta-(n+1)\beta_1+\beta_2}{\beta_2-\beta_1}.
\end{equation}
Therefore,
\begin{equation}
l_n= \left\lfloor \frac{-\scal{x_0}{u}+\beta-(n+1)\beta_1+\beta_2}{\beta_2-\beta_1} \right\rfloor
\end{equation}
and 
\begin{equation}
n- l_n= -\left\lfloor \frac{-\scal{x_0}{u}+\beta-\beta_1-(n-1)\beta_2}{\beta_2-\beta_1} \right\rfloor,
\end{equation}
which imply \eqref{e:xu}. 

To get \eqref{e:x} and \eqref{e:kn}, we distinguish two cases.

\emph{Case 1}: $\scal{x_n}{u}\leq \beta+ \beta_2$. 
On the one hand, by \eqref{e:xnS1S2} 
we must have $x_n\in S_1$ and $k(n)= 1$. 
On the other hand, from $\scal{x_n}{u}\leq \beta+ \beta_2$ and \eqref{e:xu}, noting that $\beta_1< 0$, we obtain that
\begin{subequations}
\begin{align}
\left\lfloor \frac{-\scal{x_0}{u}+\beta-(n+1)\beta_1+\beta_2}{\beta_2-\beta_1} \right\rfloor 
&\leq \frac{-\scal{x_0}{u}+\beta-n\beta_1+\beta_2}{\beta_2-\beta_1} \\
&< \frac{-\scal{x_0}{u}+\beta-(n+1)\beta_1+\beta_2}{\beta_2-\beta_1} 
\end{align}
\end{subequations}
which yields
\begin{equation}
\left\lfloor \frac{-\scal{x_0}{u}+\beta-n\beta_1+\beta_2}{\beta_2-\beta_1} \right\rfloor
= \left\lfloor \frac{-\scal{x_0}{u}+\beta-(n+1)\beta_1+\beta_2}{\beta_2-\beta_1} \right\rfloor,
\end{equation}
hence \eqref{e:x} and \eqref{e:kn} hold. 

\emph{Case 2}: $\scal{x_n}{u}> \beta+ \beta_2$. 
By \eqref{e:xnS1S2}, $x_n\in S_2$ and $k(n)= 2$.
Again using \eqref{e:xu} and noting that $\beta_1< 0< \beta_2$, we derive that
\begin{subequations}
\begin{align}
\left\lfloor \frac{-\scal{x_0}{u}+\beta-(n+1)\beta_1+\beta_2}{\beta_2-\beta_1} \right\rfloor 
&> \frac{-\scal{x_0}{u}+\beta-n\beta_1+\beta_2}{\beta_2-\beta_1} \\
&= \frac{-\scal{x_0}{u}+\beta-(n+1)\beta_1+\beta_2}{\beta_2-\beta_1}+ \frac{\beta_1}{\beta_2-\beta_1} \\
&> \left\lfloor \frac{-\scal{x_0}{u}+\beta-(n+1)\beta_1+\beta_2}{\beta_2-\beta_1} \right\rfloor- 1. 
\end{align}
\end{subequations}
It follows that
\begin{equation}
\left\lfloor \frac{-\scal{x_0}{u}+\beta-n\beta_1+\beta_2}{\beta_2-\beta_1} \right\rfloor
= \left\lfloor \frac{-\scal{x_0}{u}+\beta-(n+1)\beta_1+\beta_2}{\beta_2-\beta_1} \right\rfloor- 1,
\end{equation}  
and we have \eqref{e:x} and \eqref{e:kn}. The proof is complete.
\end{proof}

\begin{corollary}
\label{c:closedform}
Suppose that $\beta_1> \beta\geq -\beta_2$, that $x_0\in A$, and that $2\scal{x_0}{b_1-b_2}> \|b_1\|^2- \|b_2\|^2$. Then
\begin{equation}
\label{e:simplify}
(\forall n\in \NN)\quad \scal{x_n}{u}= n\beta_1+ \left\lfloor \frac{\beta-(n+1)\beta_1+\beta_2}{\beta_2-\beta_1} \right\rfloor (\beta_2- \beta_1) 
\end{equation}
and 
\begin{equation}
\label{e:x_simplify}
(\forall n\in \NN^*)\quad x_n= \left( (n-1)\beta_1+ \left\lfloor \frac{\beta-n\beta_1+\beta_2}{\beta_2-\beta_1} \right\rfloor (\beta_2- \beta_1) \right)u+ b_{k(n)}, 
\end{equation}
where 
\begin{equation}
\label{e:kn_simplify}
(\forall n\in \NN^*)\quad k(n)= \left\lfloor \frac{\beta-(n+1)\beta_1+\beta_2}{\beta_2-\beta_1} \right\rfloor- \left\lfloor \frac{\beta-n\beta_1+\beta_2}{\beta_2-\beta_1} \right\rfloor+ 1.
\end{equation}  
\end{corollary}
\begin{proof}
From $x_0\in A$, we have that $\scal{x_0}{u}= 0$ and also $R_Ax_0= P_Ax_0= x_0$. 
Since $2\scal{x_0}{b_1-b_2}> \|b_1\|^2- \|b_2\|^2$, 
it holds that $\|b_1-x_0\|^2< \|b_2-x_0\|^2$, which yields $P_BR_Ax_0= P_Bx_0= b_1$. 
Therefore, $k(1)= 1$, $x_1= x_0- P_Ax_0 +P_BR_Ax_0= b_1$, and $\scal{x_1}{u}= \scal{b_1}{u}= \beta_1$. 

On the other hand, it follows from $\beta_1> \beta\geq -\beta_2$
and $\beta_1< 0$ that $\beta+ \beta_2\geq 0$ and that $\beta<
\beta_1<0 \leq \beta+ \beta_2$. We deduce that $\scal{x_1}{u}=
\beta_1\in \left]\beta, \beta+ \beta_2 \right[$, which implies
that $x_1\in S_1$. Using Theorem~\ref{t:closedform}, we get
\eqref{e:simplify} for all $n\in \NN^*$. When $n= 0$, 
the right-hand side of \eqref{e:simplify} becomes
\begin{equation}
\left\lfloor \frac{\beta-\beta_1+\beta_2}{\beta_2-\beta_1} \right\rfloor (\beta_2- \beta_1) =0= \scal{x_0}{u}
\end{equation} 
since $0< \beta-\beta_1+\beta_2< \beta_2-\beta_1$. Hence, \eqref{e:simplify} holds for all $n\in \NN$, which together with the second part of Theorem~\ref{t:closedform} completes the proof.
\end{proof}

\begin{example}
\label{ex:R}
Suppose that $X= \RR$, that $A= \{0\}$, and that $B= \{b_1, b_2\}$ with $b_1= -1$ and $b_2= r$, where $r\in \RR$, $r> 1$. Let $(x_n)_\nnn$ be a DR sequence with respect to $(A, B)$ with starting point $x_0= 0$. Then
\begin{equation}
(\forall n\in \NN)\quad x_n= 
-n+ \left\lfloor \frac{n}{r+1} + \frac{1}{2} \right\rfloor (r+1).
\end{equation} 
\end{example}
\begin{proof}
Let $u= 1$. Then $A= \{u\}^\perp$ and $(\forall x\in \RR)$ $\scal{x}{u}= x$. 
We have that $\beta_1= \scal{b_1}{u}= -1< 0$, $\beta_2= \scal{b_2}{u}= r >0$, and, since $r> 1$, 
\begin{equation}
-1=\beta_1> \beta= \frac{|b_1-b_2|^2}{2(\beta_1-\beta_2)}=
-\frac{(r+1)^2}{2(r+1)}= -\frac{r+1}{2}> -\beta_2=-r.
\end{equation}  
It is clear that $x_0= 0\in A$ and that 
$2\scal{x_0}{b_1-b_2}= 0> 1- r^2= |b_1|^2- |b_2|^2$. 
Now applying Corollary~\ref{c:closedform} yields
\begin{equation}
(\forall n\in \NN)\quad x_n= \scal{x_n}{u}= -n+ \left\lfloor \frac{-\frac{r+1}{2}+(n+1)+r}{r+1} \right\rfloor (r+1), 
\end{equation} 
and the conclusion follows.
\end{proof}

\begin{example}
\label{ex:R2} 
Suppose that $X= \RR^2$, that $A= \RR\times \{0\}$, and that $B= \{b_1, b_2\}$ with $b_1= (0, -1)$ and $b_2= (1, r)$, where $r\in \RR$, $r\geq \sqrt{2}$. Let $(x_n)_\nnn$ be a DR sequence with respect to $(A, B)$ with starting point $x_0= (\alpha, 0)$, where $\alpha\in \RR$, $\alpha< r^2/2$. 
Then $(\forall n\in \NN^*)$:
\begin{multline}
x_n= \left(\left\lfloor \frac{n}{r+1}+ \frac{r^2+2r}{2(r+1)^2} \right\rfloor- \left\lfloor \frac{n-1}{r+1}+ \frac{r^2+2r}{2(r+1)^2} \right\rfloor, -n+ \left\lfloor \frac{n}{r+1}+ \frac{r^2+2r}{2(r+1)^2} \right\rfloor (r+1) \right).
\end{multline} 
\end{example}
\begin{proof}
In this case, $A= \{u\}^\perp$ with $u= (0, 1)$, $\beta_1= \scal{b_1}{u}= -1< 0$, $\beta_2= \scal{b_2}{u}= r >0$, and 
\begin{equation}
\beta_1= -1> \beta= \frac{\|b_1-b_2\|^2}{2(\beta_1-\beta_2)}= -\frac{1+ (r+1)^2}{2(r+1)}= -1- \frac{r^2}{2(r+1)}. 
\end{equation} 
On the one hand, $\beta+ \beta_2= \frac{r^2-2}{2(r+1)}\geq 0$. 
On the other hand, it is straightforward to see that $x_0\in A$
and that $2\scal{x_0}{b_1-b_2} = -2\alpha>-r^2= \|b_1\|^2 -\|b_2\|^2$. 
Applying Corollary~\ref{c:closedform}, we obtain that
\begin{subequations}
\begin{align}
(\forall n\in \NN^*)\quad \scal{x_n}{u}&= -n+ \left\lfloor \frac{-1-\frac{r^2}{2(r+1)}+(n+1)+r}{r+1} \right\rfloor (r+1) \\
&= -n+ \left\lfloor \frac{n}{r+1}+ \frac{r^2+2r}{2(r+1)^2} \right\rfloor (r+1). 
\end{align}
\end{subequations}
Now for each $n\in \NN^*$, writing $x_n= (\alpha_n, \beta_n)\in \RR^2$, we observe that $\beta_n= \scal{x_n}{u}$ and, by \eqref{e:x_simplify}, $\alpha_n$ is actually the first coordinate of $b_{k(n)}$, that is,
\begin{equation}
\alpha_n= \begin{cases}
0 &\text{if~} k(n)= 1, \\
1 &\text{if~} k(n)= 2,
\end{cases}
\end{equation} 
which combined with \eqref{e:kn_simplify} implies that
\begin{equation}
\alpha_n= k(n)- 1= \left\lfloor \frac{n}{r+1}+ \frac{r^2+2r}{2(r+1)^2} \right\rfloor- \left\lfloor \frac{n-1}{r+1}+ \frac{r^2+2r}{2(r+1)^2} \right\rfloor. 
\end{equation}
The conclusion follows.
\end{proof}

Let us specialize Example~\ref{ex:R} further and
also illustrate Theorem~\ref{t:2points}. 

\begin{example}[rational case]
Suppose that $X= \RR$, that $A= \{0\}$, and that $B= \{-1,
{2}\}$. Let $(x_n)_\nnn$ be a DR sequence with respect to $(A, B)$ with starting point $x_0= 0$. Then
\begin{equation}
(\forall n\in \NN)\quad x_n= -n+ 3 \left\lfloor
\frac{n}{3}+\frac{1}{2} \right\rfloor
\end{equation} 
and $(x_n)_\nnn = \big(
0,
-1,
1,
0,
-1,
1,
0,
-1,
1,
\ldots
\big)$ is periodic. (See also \cite[Remark~6]{BN14} for another
cyclic example.)
\end{example}
\begin{proof}
Apply Example~\ref{ex:R} with $b_1=-1$ and $b_2=2$. 
\end{proof}

\begin{example}[irrational case]
Suppose that $X= \RR$, that $A= \{0\}$, and that $B= \{-1,
\sqrt{2}\}$. Let $(x_n)_\nnn$ be a DR sequence with respect to $(A, B)$ with starting point $x_0= 0$. Then
\begin{equation}
(\forall n\in \NN)\quad x_n= -n+ \left\lfloor
\frac{n}{\sqrt{2}+1}+\frac{1}{2} \right\rfloor (\sqrt{2}+1)
\end{equation} 
and $(x_n)_\nnn = \big(
0,
-1,
-1+\sqrt{2},
-2+\sqrt{2},
-2+2\sqrt{2},
-3+2\sqrt{2},
-4+2\sqrt{2},
-4+3\sqrt{2},
\ldots
\big)$ which is not periodic. 
\end{example}
\begin{proof}
Apply Example~\ref{ex:R} with $b_1=-1$ and $b_2=\sqrt{2}$. 
\end{proof}

\begin{remark}
Some comments on the last examples are in order.
\begin{enumerate}
\item
We note that the last examples feature terms 
resembling (inhomogeneous) Beatty sequences; see \cite{Hav12}. 
In fact, let us disclose that we started this journey 
by experimentally investigating Example~\ref{ex:R2} 
which eventually led to the more general analysis in this paper.
Specifically, in Example~\ref{ex:R2}, if $r= \sqrt{2}$, then $x_n= (u_n, -v_n+w_n\sqrt{2})$, 
where the integer sequences
\begin{subequations}
\begin{align}
u_n&:= \lfloor (n+1)(\sqrt{2}-1) \rfloor- \lfloor n(\sqrt{2}-1) \rfloor= \lfloor (n+1)\sqrt{2} \rfloor- \lfloor n\sqrt{2} \rfloor- 1,\\ 
v_n&:= n-\lfloor (n+1)(\sqrt{2}-1) \rfloor= \lfloor (n+1)(2-\sqrt{2}) \rfloor,\\ 
w_n&:= \lfloor (n+1)(\sqrt{2}-1) \rfloor= \lfloor (n+1)\sqrt{2} \rfloor- n- 1
\end{align}
\end{subequations}
are respectively listed as \cite{OEIS_A188037}, \cite{OEIS_A074840}, and \cite{OEIS_A097508} 
(shifted by one) in the \emph{On-Line Encyclopedia of Integer Sequences}. 
\item
Finally, let us contrast the DR algorithm to the method of
alternating projections (see, e.g., \cite{BB96} and \cite{BC17})
in the setting of Example~\ref{ex:R}: indeed, the 
sequence $(x_0,P_Ax_0,P_BP_Ax_0,\ldots)$ is simply
$(0,0,-1,0,-1,0,\ldots)$ regardless of whether or not $r>1$ is irrational. 
It was also suggested in \cite{BDNP16a} that, 
for the convex feasibility problem, the DR algorithm 
outperforms the method of alternating projections 
in the absence of constraint qualifications.
\end{enumerate}
\end{remark}

\section{Conclusion}
\label{s:conclusion}

In this paper, we provided a detailed analysis of the
Douglas--Rachford algorithm for the case when one set is a
hyperplane and the other a doubleton. We
characterized cycling of this method in terms of the ratio of 
the distances of the points to the hyperplane. Moreover, we
presented closed-form expressions of the actual iterates. The
results obtained show the surprising complexity of this 
algorithm
when compared to, e.g., the method of alternating projections.

\subsection*{Acknowledgments}
HHB was partially supported by the Natural Sciences and
Engineering Research Council of Canada.
MND was partially supported by the Australian Research Council.

\end{document}